\documentclass[reqno,12pt]{amsart}
\usepackage{amsthm,amsfonts,amssymb,euscript,
mathrsfs,graphics,color,amsmath,latexsym,marginnote}
\usepackage{dsfont}   
\usepackage{mathtools}

\theoremstyle{plain}

\usepackage{amsmath}
\usepackage{hyperref}
\usepackage{times}

\numberwithin{equation}{section}

\newtheorem{theorem}{Theorem}[section]
\newtheorem{proposition}[theorem]{Proposition}
\newtheorem{lemma}[theorem]{Lemma}

\newtheorem{remark}[theorem]{Remark}
\newtheorem{remarks}[theorem]{Remark}

\newtheorem{definition}[theorem]{Definition}
\newcommand{\be}{\begin{equation}}
\newcommand{\ee}{\end{equation}}

\newcommand{\R}{\mathbb R}
\newcommand{\C}{\mathbb C}

\newcommand{\Z}{\mathbb Z}

\newcommand{\N}{\mathbb N}
\newcommand{\T}{\mathbb T}

\newcommand{\s }{\sigma }
\newcommand{\ii }{{\rm i} }

\newcommand{\x }{\xi }

\newcommand{\opw}{{Op^{\mathrm{W}}}}
\newcommand{\opbw}{{Op^{{\scriptscriptstyle{\mathrm BW}}}}}

\def\hat{\widehat}

\def\bar{\overline}

\def\ba{\begin{aligned}}
\def\ea{\end{aligned}}
\def\beginm{\begin{multline}}
\def\endm{\end{multline}}

\makeatletter

\def\l@subsection{\@tocline{2}{0pt}{2.5pc}{5pc}{}}
\def\l@subsubsection{\@tocline{3}{0pt}{4.5pc}{5pc}{}}
\renewcommand\tocchapter[3]{%
  \indentlabel{\@ifnotempty{#2}{\ignorespaces#2.\quad}}#3%
}
\newcommand\@dotsep{4.5}
\def\@tocline#1#2#3#4#5#6#7{\relax
  \ifnum #1>\c@tocdepth 
  \else
    \par \addpenalty\@secpenalty\addvspace{#2}%
    \begingroup \hyphenpenalty\@M
    \@ifempty{#4}{%
      \@tempdima\csname r@tocindent\number#1\endcsname\relax
    }{%
      \@tempdima#4\relax
    }%
    \parindent\z@ \leftskip#3\relax \advance\leftskip\@tempdima\relax
    \rightskip\@pnumwidth plus1em \parfillskip-\@pnumwidth
    #5\leavevmode\hskip-\@tempdima{#6}\nobreak
    \leaders\hbox{$\m@th\mkern \@dotsep mu\hbox{.}\mkern \@dotsep mu$}\hfill
    \nobreak
    \hbox to\@pnumwidth{\@tocpagenum{#7}}\par
    \nobreak
    \endgroup
  \fi}
\def\l@subsection{\@tocline{2}{0pt}{2.5pc}{5pc}{}}

\begin{document}

\title{On the Cauchy problem for  quasi-linear Hamiltonian KdV-type equations.}

\author{Felice Iandoli}
\address{\scriptsize{Laboratoire Jacques-Louis Lions, Sorbonne Universit\'e, UMR CNRS 7598\\
4, Place Jussieu\\
 75005 Paris Cedex 05, France}}
\email{felice.iandoli@sorbonne-universite.fr}




\thanks{The author has been supported by ERC grant ANADEL 757996.}

\begin{abstract}
We prove local in time well-posedness for a class of quasilinear Hamiltonian KdV-type equations with periodic boundary conditions, more precisely we show existence, uniqueness and continuity of the solution map. We improve the previous result in \cite{Mietka}, generalising the considered class of  equations and improving the regularity assumption on the initial data.
\end{abstract}


\maketitle

\tableofcontents

\section{Introduction}
In this paper 
$u(t,x)$ is a function of time $t\in [0,T)$, $T>0$ and space $x\in \mathbb{T}:=\mathbb{R}/2\pi\mathbb{Z}$. $F(x,z_0,z_1)$ is a polynomial function such that $F(x,0,z_1)=F(x,z_0,0)=\partial_{z_0}F(x,0,z_1)=\partial_{z_1}F(x,z_0,0)=0$. Throughout the paper we shall assume that there exists a constant $\mathfrak{c}>0$ such that
\begin{equation}\label{ellit}
\partial_{z_1z_1}^2 F(x,z_0,z_1)\geq \mathfrak{c},
\end{equation}
for any $x\in\T$, $z_0,$ $z_1\in\R$.
We shall denote the partial derivatives of the function $u$ by $u_t, u_x, u_{xx}$ and  $ u_{xxx}$, by $\partial_x, \partial_{z_0}, \partial_{z_1}$ the partial derivatives of the function $F$ and by $\frac{d}{dx}$ the total derivative with respect to the variable $x$. For instance we have $\frac{d}{dx}F(x,u,u_{x})=\partial_{x}F(x,u,u_x)+\partial_{z_0}F(x,u,u_x)u_x+\partial_{z_1}F(x,u,u_x)u_{xx}.$ We consider the equation
\begin{equation}\label{KdV-astratta}
u_t=\frac{d}{dx}\Big(\nabla_uH(x,u,u_x)\Big), \quad H(x,u,u_x):=\int_{\T}F(x,u,u_x)dx,
\end{equation}
where we denoted  by $\nabla_{u}H$ the $L^2$-gradient of the  Hamiltonian function $H(x,u,u_x)$ on the phase space
\begin{equation}
H_0^s(\T):=\{u(x)\in H^s(\T):\int_{\T}u(x)dx=0\}
\end{equation}
endowed with the non-degenerate symplectic form $\Omega(u,v):=\int_{{\T}}(\partial_x^{-1}u)vdx$ ($\partial_x^{-1}$ is the periodic primitive of $u$ with zero average) and with the norm $\|u\|_{H^s}:=\sum_{j\in\Z^*}|u_j|^{2}|j|^2$ ($u_j$ are the Fourier coefficients of the periodic function $u$).
The main result is the following.
\begin{theorem}\label{totale}
Let $s>4+1/2$ and assume \eqref{ellit}. Then for any $u_0\in H^s_0(\T)$ there exists a time $T:=T(\|u_0\|_{H^s})$ and a unique solution of \eqref{KdV-astratta} with initial condition $u(0,x)=u_0(x)$ satisfying $u(t,x)\in C^0([0,T),H_0^s(\T))\cap C^1([0,T),H_0^{s-3}(\T)).$ Moreover the solution map $u_0(x)\mapsto u(t,x)$ is continuous with respect to the $H^s_0$ topology for any $t$ in $[0,T)$.
\end{theorem}
This theorem improves the previous one in \cite{Mietka} by Mietka. The result in such a paper holds true if the Hamiltonian function has the form $H(u)$, while here we allow the explicit dependence on the $x$ variable (non autonomous equation) and the dependence on $u_x$. We tried to optimise our result in terms of regularity of the initial condition, we do not  know if the result is improvable. If we apply our method to the equation considered by Mietka, we find a local well-posedness theorem if the initial condition belongs to the space $H_0^s$ with $s>3+1/2$ (which is natural since the nonlinearity may contain up to three derivatives of $u$), while in \cite{Mietka} one requires $s\geq4$. In our statement we need  $s>4+1/2$ because  our equation is more general and we have the presence of one more derivative in the coefficients with respect to the equation considered in \cite{Mietka}.\\
The proof of Theorem \ref{totale} is an application of a method which has been developed in \cite{FI1, FI2} and then improved, in terms of regularity of initial condition, in \cite{BMM1}. Here we follow closely the method in \cite{BMM1} and we use several results proven therein.
Both the schemes, the one used in \cite{Mietka} and in \cite{FI1,FI2,BMM1}, rely on solely energy method, the second one is slightly more refined because of the use of para-differential calculus which allows us to work in fractional Sobolev spaces and to treat more general nonlinear terms. The main idea is to introduce a convenient energy, which is equivalent to the Sobolev norm,  which commutes with the principal (quasi-linear)  term in the equation (see \eqref{energia-modi}). 
In \cite{FI1,FI2,BMM1} the main difficulty comes from the fact that, after the paralinearization, one needs to prove \emph{a priori} estimates on a system of coupled equations. One needs then to decouple the equations through convenient changes of coordinates which are used to define the modified energy. In the case of KdV equation \eqref{KdV-astratta}, we have a scalar equation with the sub-principal symbol which is \emph{real} (and so it defines a selfadjoint operator), see \eqref{KdVpara}, therefore it is impossible to obtain energy estimates directly. This term may be completely removed (see Lemma \ref{key}) thanks to the Hamiltonian structure.
For similar constructions of such kind of energies one can look also at \cite{IP, FI2, BMM1, FGI20,FIM}.\\
The general equation \eqref{KdV-astratta} contains the ``classical" KdV equation $u_t+uu_x+u_{xxx}=0$ and the \emph{modified} KdV $u_t+u^pu_x+u_{xxx}=0$, $p\geq 2$. Obviously, for the last two equations better results may be obtained, concerning KdV we quote Bona-Smith \cite{BSkdv}, Kato \cite{kato}, Bourgain \cite{Bourgain}, Kenig-Ponce-Vega \cite{KPV17,KPV19}, Christ-Colliander-Tao \cite{CCT}. For the general equation, as the one considered in this paper here, several results have been proven by Colliander-Keel-Staffilani-Takaoka-Tao \cite{CKSTT},  Kenig-Ponce-Vega \cite{KPV18} and the aforementioned Mietka \cite{Mietka}.\\

\section{Paradifferential calculus}
In this section we recall some results concerning the para-differential calculus, we follow \cite{BMM1}. We introduce the Japanese bracket $\langle{\xi}\rangle=\sqrt{1+\xi^2}.$
\begin{definition}
Given $m,s\in\mathbb{R}$ we denote by $\Gamma^m_s$ the space of functions $a(x,\xi)$ defined on $\T\times \R$ with values in $\C$, which are $C^{\infty}$ with respect to the variable $\xi\in\R$ and such that for any $\beta\in \N\cup\{0\}$,
there exists a constant $C_{\beta}>0$ such that 
\begin{equation}\label{stima-simbolo}
\|\partial_{\xi}^{\beta} a(\cdot,\xi)\|_{H_0^s}\leq C_{\beta}\langle\xi\rangle^{m-\beta}, \quad \forall \xi\in\R.
\end{equation}
\end{definition}
We endow the space $\Gamma^m_{s}$ with the family of norms 
\begin{equation}\label{seminorme}
|a|_{m,s,n}:=\max_{\beta\leq n} \sup_{\xi\in\R}\|\langle\xi\rangle^{\beta-m}a(\cdot,\xi)\|_{H^s_0}.
\end{equation}
Analogously for a given Banach space $W$ we denote by $\Gamma^m_{W}$ the space of functions which verify the \eqref{stima-simbolo} with the $W$-norm instead of $H^s_0$, we also denote by $|a|_{m,W,n}$ the $W$ based seminorms \eqref{seminorme} with $H^s_0\rightsquigarrow W$.\\
We say that a symbol $a(x,\xi)$ is spectrally localised if there exists $\delta>0$ such that $\hat{a}(j,\xi)=0$ for any $|j|\geq \delta \langle\xi\rangle.$\\
Consider a function $\chi\in C^{\infty}(\R,[0,1])$ such that $\chi(\xi)=1$ if $|\xi|\leq 1.1$ and $\chi(\xi)=0$ if $|\xi|\geq 1.9$. Let $\epsilon\in(0,1)$ and define moreover $\chi_{\epsilon}(\xi):=\chi(\xi/\epsilon).$ Given $a(x,\xi)$ in $\Gamma^m_s$ we define the regularised symbol
\begin{equation*}
a_{\chi}(x,\xi):=\sum_{j\in \Z^*}\hat{a}(j,\xi)\chi_{\epsilon}(\tfrac{j}{\langle\xi\rangle})e^{\ii jx}.
\end{equation*}
 For a symbol $a(x,\x)$ in $\Gamma^m_s$
we define its Weyl and Bony-Weyl
quantization as 
\begin{equation}\label{quantiWeylcl}
\opw(a(x,\xi))h:=\frac{1}{(2\pi)}\sum_{j\in \mathbb{Z}^*}e^{\ii j x}
\sum_{k\in\mathbb{Z}^*}
\hat{a}\big(j-k,\frac{j+k}{2}\big)\widehat{h}(k),
\end{equation}
\begin{equation}\label{quantiWeyl}
\opbw(a(x,\xi))h:=\frac{1}{(2\pi)}\sum_{j\in \mathbb{Z}^*}e^{\ii j x}
\sum_{k\in\mathbb{Z}^*}
\chi_{\epsilon}\Big(\frac{|j-k|}{\langle j+k\rangle}\Big)
\hat{a}\big(j-k,\frac{j+k}{2}\big)\widehat{h}(k).
\end{equation}
We list below a series of theorems and lemmas that will be used in the paper. All the statements have been taken from \cite{BMM1}. The first one is a result concerning the action of a paradifferential operator  on Sobolev spaces. This is Theorem 2.4 in \cite{BMM1}.
\begin{theorem}\label{azione}
Let $a\in\Gamma^m_{s_0}$, $s_0>1/2$ and $m\in\R$. Then $\opbw(a)$ extends as a bounded operator from $H^{s-m}_0(\T)$ to $H^s_0(\T)$ for any $s\in\R$ with estimate 
\begin{equation}\label{ac1}
\|\opbw(a)u\|_{H^{s-m}_0}\lesssim |a|_{m,s_0,4}\|u\|_{H^s_0},
\end{equation}
for any $u$ in $H^s_0(\T)$. Moreover for any $\rho\geq 0$ we have for any $u\in H^s_0(\T)$
\begin{equation}\label{ac2}
\|\opbw(a)u\|_{H^{s-m-\rho}_0}\lesssim |a|_{m,s_0-\rho,4}\|u\|_{H^s_0}.
\end{equation}
\end{theorem}
We now state a result regarding symbolic calculus for the composition of Bony-Weyl paradifferential operators. In the rest of the section, since there is no possibility of confusion, we shall denote the total derivative $\tfrac{d}{dx}$ as $\partial_x$  with the aim of  improving the readability of the formul\ae.   Given two symbols $a$ and $b$ belonging to $\Gamma^m_{s_0+\rho}$ and $\Gamma^{ m'}_{s_0+\rho}$ respectively, we define for $\rho\in(0,3]$
\begin{equation}\label{cancelletto}
a\#_{\rho} b= \begin{cases}
ab \quad \rho\in(0,1]\\
ab+\frac{1}{2\ii}\{a,b\}\quad \rho\in (1,2],\\
ab+\frac{1}{2\ii}\{a,b\}-\frac18\mathfrak{s}(a,b) \quad \rho\in (2,3],\\
\end{cases}
\end{equation}
where we denoted by $\{a,b\}:=\partial_{\xi}a\partial_xb-\partial_xa\partial_{\xi}b$ the Poisson's bracket between symbols and $\mathfrak{s}(a,b):=\partial_{xx}^2a\partial_{\x\x}^2b-2\partial_{x\x}^2a\partial_{x\x}^2b+\partial_{\x\x}^2a\partial_{xx}^2b$. 
\begin{remark}\label{simmetrie}
According to the notation above we have $ab\in\Gamma^{m+m'}_{s_0+\rho}$, $\{a,b\}\in\Gamma^{m+m'-1}_{s_0+\rho-1}$ and $\mathfrak{s}(a,b)\in\Gamma^{m+m'-2}_{s_0+\rho-2}$. Moreover $\{a,b\}=-\{b,a\}$ and $\mathfrak s(a,b)=\mathfrak s(b,a).$
\end{remark}
The following is essentially Theorem 2.5 of \cite{BMM1}, we just need some more precise symbolic calculus since we shall deal with nonlinearities containing three derivatives, while in \cite{BMM1} they have nonlinearities with two derivatives. 
\begin{theorem}\label{compo}
Let $a\in\Gamma^m_{s_0+\rho}$ and $b\in\Gamma^{m'}_{s_0+\rho}$ with $m,m'\in\mathbb{R}$ and $\rho\in(0,3]$. We have $\opbw(a)\circ\opbw(b)=\opbw(a\#_{\rho}b)+R^{-\rho}(a,b)$, where the linear operator $R^{-\rho}$ is defined on $H^s_0(\T)$ with values in $H^{s+\rho-m-m'}$, for any $s\in\R$ and it satisfies 
\begin{equation}\label{resto}
\|R^{-\rho}(a,b)\|_{H^{s-(m+m')+\rho}_0}\lesssim (|a|_{m,s_0+\rho,N}|b|_{m',s_0,N}+|a|_{m,s_0,N}|b|_{m',s_0+\rho,N})\|u\|_{H^s_0},
\end{equation}
where $N\geq 8$.\end{theorem}
\begin{proof}
We prove the statement for $\rho\in (2,3]$, for smaller $\rho$ the reasoning is similar.
Recalling formul\ae \eqref{quantiWeyl} and \eqref{quantiWeylcl} we have
\begin{equation*}\begin{aligned}
\opbw({a})\opbw({b})u&=\opw(a_{\chi})\opw(b_{\chi})u\\
&=\sum_{j,k,\ell}\hat{a}_{\chi}(j-k,\frac{j+k}{2})\hat{b}_{\chi}(k-\ell,\frac{k+\ell}{2})u_{\ell}e^{\ii jx}.
\end{aligned}\end{equation*}
We Taylor expand $\hat{a}_{\chi}(j-k,\frac{j+k}{2})$ with respect to the second variable in the point $\tfrac{j+\ell}{2}$, we have
\begin{equation*}\begin{aligned}
\hat{a}_{\chi}&(j-k,\tfrac{j+k}{2})=\\
&\hat{a}_{\chi}(j-k,\tfrac{j+\ell}{2})+\tfrac{k-\ell}{2}\partial_{\xi}\hat{a}_{\chi}(j-k,\tfrac{j+\ell}{2})
+\tfrac{(k-\ell)^2}{8}\hat{a}_{\chi}(j-k,\tfrac{j+\ell}{2})\\
&+\tfrac{(k-\ell)^3}{8}\int_0^1(1-t)^2\partial_{\xi}^3\hat{a}_{\chi}(j-k,\tfrac{j+\ell+t(k-\ell)}{2})dt.
\end{aligned}\end{equation*}
Analogously we obtain
\begin{equation*}
\begin{aligned}
\hat{b}_{\chi}&(k-\ell,\tfrac{k+\ell}{2})=\\
&+\tfrac{k-j}{2}\partial_{\xi}\hat{b}_{\chi}(k-\ell,\tfrac{j+\ell}{2})+\tfrac{(k-j)^2}{8}\partial_{\xi}^2\hat{b}_{\chi}(k-\ell,\tfrac{j+\ell}{2})\\
&+\tfrac{(k-j)^3}{8}\int_0^1(1-t)^2\partial_{\xi}^3\hat{b}_{\chi}(k-\ell,\tfrac{j+\ell+t(k-j)}{2})dt.
\end{aligned}
\end{equation*}
An explicit computation proves that 
\begin{equation*}
\begin{aligned}
\opbw(a)\opbw(b)-\opbw(ab+\tfrac{1}{2\ii}-\tfrac18\mathfrak{s}(a,b))u=\sum_{j=1}^{4}R_i(a,b)u,
\end{aligned}
\end{equation*}
where 
\begin{align*}
R_1(a,b)&:=\opw\big(a_{\chi}b_{\chi}-(ab)_{\chi}+\tfrac{1}{2\ii}(\{a_{\chi},b_{\chi}\}-\{a,b\}_{\chi})-\tfrac18(\mathfrak{s}(a_{\chi},b_{\chi})-\mathfrak{s}(a,b)_{\chi})\big)u,\\
R_2(a,b)&:=\sum Q_3^b\big(\hat{a}_{\chi}(j-k,\tfrac{j+\ell}{2})+\tfrac{k-\ell}{2}\partial_{\xi}\hat{a}_{\chi}(j-k,\tfrac{j+\ell}{2})
+\tfrac{(k-\ell)^2}{8}\hat{a}_{\chi}(j-k,\tfrac{j+\ell}{2})\big)u_{\ell}e^{\ii j x},\\
R_3(a,b)&:=\sum Q_3^a \hat{b}_{\chi}(k-\ell,\tfrac{k+\ell}{2})u_{\ell}e^{\ii j x},\\
R_4(a,b)&:=-\tfrac{1}{16\ii}\opw(\partial_x^2\partial_{\xi}a\partial_x\partial^2_{\xi}b+\partial_x^2\partial_{\xi}b\partial_x\partial_{\xi}^2a)u+\tfrac{1}{64}\opw(\partial_x^2\partial_{\xi}^2a\partial_x^2\partial_{\xi}^2b)u,
\end{align*}
where we have defined $Q_3^a:=\tfrac{(k-\ell)^3}{8}\int_0^1(1-t)^2\partial_{\xi}^3\hat{a}_{\chi}(j-k,\tfrac{j+\ell+t(k-\ell)}{2})dt$ and analogously $Q_3^b$. We prove that each $R_i$ fulfils the estimate \eqref{resto}. The remainders $R_1$, $R_2$ and $R_3$ have to be treated as done in the proof of Theorem 2.5 in \cite{BMM1}, we just underline the differences. Concerning $R_1$ it is enough to prove that for any $\alpha\leq 2$ the symbol $\partial_\xi^{\alpha}a_{\chi}\partial_x^{\alpha}b_{\chi}-\partial_\xi^{\alpha}b_{\chi}\partial_x^{\alpha}a_{\chi}$ is a spectrally localised symbol belonging  to $\Gamma^{m+m'-\rho}_{L^{\infty}}$. Following word by word the proof in \cite{BMM1}, with $d=1$ and $\alpha=2$ (instead of $\alpha=1$ therein) one may bound $|\partial_\xi^{\alpha}a_{\chi}\partial_x^{\alpha}b_{\chi}-\partial_\xi^{\alpha}b_{\chi}\partial_x^{\alpha}a_{\chi}|_{m,W^{1,\infty},n}\lesssim |a|_{m,W^{1,\infty},n+2}|b|_{m',L^{\infty},n+2}+|a|_{m,L^{\infty},n+2}|b|_{m',W^{1,\infty},n+2}$. The estimate \eqref{resto} on the remainder $R_1$ follows from Theorem A.7 in \cite{BMM1}. In order to prove that $R_3$ and $R_2$ satisfy \eqref{resto}, one has to follow the proof of Theorem A.5 in \cite{BMM1} with $d=1$, $\alpha=3$ and $\beta\leq 2$ corresponding to the remainder $R_2(a,b)$ therein. Concerning the remainder $R_4$ we have the following: the symbol of the first summand is in the class $\Gamma^{m+m'-3}_{s_0}$ and the second in $\Gamma^{m+m'-4}_{s_0}$, the estimates follow then by Theorem \ref{azione}.
\end{proof}

\begin{lemma}{\bf (Paraproduct).}\label{Paraproduct}
Fix $s_0>1/2$ and 
let $f,g\in H^{s}_0(\mathbb{T};\mathbb{C})$ for $s\geq s_0$. Then
\begin{equation}\label{eq:paraproduct}
fg=\opbw(f)g+\opbw({g})f+\mathcal{R}(f,g)\,,
\end{equation}
where
\begin{equation}\label{eq:paraproduct2}
\begin{aligned}
\widehat{\mathcal{R}(f,g)}(\x)&=\frac{1}{(2\pi)}
\sum_{\eta\in \mathbb{Z}^*}
a(\x-\eta,\xi)\hat{f}(\x-\eta)\hat{g}(\eta)\,,\\
 |a(v,w)|&\lesssim\frac{(1+\min(|v|,|w|))^{\rho}}{(1+\max(|v|,|w|))^{\rho}}\,,
\end{aligned}
\end{equation}
for any $\rho\geq0$.
For $0\leq \rho\leq s-s_0$ one has
\begin{equation}\label{eq:paraproduct22}
\|\mathcal{R}(f,g)\|_{H^{s+\rho}_0}\lesssim\|f\|_{H^{s}_0}\|g\|_{H^{s}_0}\,.
\end{equation}
\end{lemma}

\begin{proof}
Notice that
\begin{equation}\label{prodFG}
\widehat{(fg)}(\x)=\sum_{\eta\in\mathbb{Z}^*}\hat{f}(\x-\eta)\hat{g}(\eta)\,.
\end{equation}
Consider the cut-off function $\chi_{\epsilon}$ 
and define a new cut-off function $\Theta : \mathbb{R} \to [0,1]$ as
\begin{equation}\label{cutoffTHETA}
1=\chi_{\epsilon}\left(\frac{| \x-\eta|}{\langle\x+\eta\rangle}\right)+
\chi_{\epsilon}\left(\frac{|\eta|}{\langle2\x-\eta\rangle}\right)+\Theta(\x,\eta)\,.
\end{equation}
Recalling \eqref{prodFG} and \eqref{quantiWeyl} we note that 
\begin{equation}\label{paraprodFG}
\begin{aligned}
\widehat{(T_fg)}(\x)&=\sum_{\eta\in\mathbb{Z}^{*}}
\chi_{\epsilon}\left(\frac{|\x-\eta|}{\langle\x+\eta\rangle}\right)
\hat{f}(\x-\eta)\hat{g}(\eta)\,, \\
\widehat{(T_g f)}(\x)&=\sum_{\eta\in\mathbb{Z}^{*}}
\chi_{\epsilon}\left(\frac{|\eta|}{\langle2\x-\eta\rangle}\right)
\hat{f}(\x-\eta)\hat{g}(\eta)\,, 
\end{aligned}
\end{equation}
and 
\begin{equation}\label{pararestoFG}
\mathcal{R}:=\mathcal{R}(f,g)\,,\qquad
\widehat{\mathcal{R}}(\x):=
\sum_{\eta\in\mathbb{Z}^{*}}
\Theta(\x,\eta)
\hat{f}(\x-\eta)\hat{g}(\eta)\,.
\end{equation}
To obtain the second in \eqref{paraprodFG} one has to use the \eqref{quantiWeyl}
and perform the change of variable $\x-\eta\rightsquigarrow \eta$.
By the definition of the cut-off function $\Theta(\x,\eta)$ we deduce
that, if $\Theta(\x,\eta)\neq0$ we must have
\begin{equation}\label{condTheta}
|\x-\eta|\geq \frac{5\epsilon}{4}\langle\x+\eta\rangle\quad {\rm and}
\quad 
|\eta|\geq \frac{5\epsilon}{4}\langle2\x-\eta\rangle
\qquad \Rightarrow\quad
\langle \eta\rangle \sim \langle \x-\eta\rangle\,.
\end{equation}
This implies that, setting 
$a(\x-\eta,\eta):=\Theta(\x,\eta)$, we get the \eqref{eq:paraproduct2}.
The \eqref{condTheta} also implies that
$\langle\x\rangle\lesssim
\max\{
\langle \x-\eta\rangle,\langle\eta\rangle\}$.
Then we have
\[
\begin{aligned}
\|\mathcal{R}h\|_{H^{s+\rho}_0}^{2}&\lesssim\sum_{\x\in\mathbb{Z}^{*}}\Big(
\sum_{\eta\in\mathbb{Z}^{*}}
|a(\x-\eta,\eta)||\hat{f}(\x-\eta)||\hat{g}(\eta)|
|\x|^{s+\rho}
\Big)^{2}\\
&\stackrel{\mathclap{\eqref{eq:paraproduct2}}}{\lesssim}
\sum_{\x\in\mathbb{Z}^{*}}\Big(
\sum_{|\x-\eta|\geq |\eta|}
|\x-\eta|^{s}|\hat{f}(\x-\eta)| | \eta|^{\rho}|\hat{g}(\eta)|
\Big)^{2}
\\&+
\sum_{\x\in\mathbb{Z}^{*}}\Big(
\sum_{|\x-\eta|\leq |\eta|}
| \x-\eta|^{\rho}|\hat{f}(\x-\eta)||\hat{g}(\eta)|
|\eta|^{s}
\Big)^{2}\\
&\lesssim
\sum_{\x,\eta\in\mathbb{Z}^{*}}
| \eta|^{2(s_0+\rho)}|\hat{g}(\eta)|^{2}
|\x-\eta|^{2s}|\hat{f}(\x-\eta)|^{2}\\
&+\sum_{\x,\eta\in\mathbb{Z}^{*}}
| \eta|^{2s}|\hat{g}(\eta)|^{2}
|\x-\eta|^{2(s_0+\rho)}|\hat{f}(\x-\eta)|^{2}\\
&\lesssim\|f\|_{H^{s}_0}^{2}\|g\|_{H^{s_0+\rho}_0}^{2}
+\|f\|_{H^{s_0+\rho}_0}^{2}\|g\|_{H^{s}_0}^{2}\,,
\end{aligned}
\]
which implies the \eqref{eq:paraproduct22} for  $s_0+\rho\leq s$.
\end{proof}

\section{Paralinearization}
The equation \eqref{KdV-astratta} is equivalent to 
\begin{equation}\label{KdV-concreta}
\begin{aligned}
u_t&+u_{xxx}\partial_{z_1z_1}^2 F+2u_{xx}\partial^3_{z_1xz_1}F+u_{xx}^2\partial_{z_1z_1z_1}^3F
     +2u_xu_{xx}\partial_{z_1z_1z_0}^3F\\
     &+u_x^2\partial_{z_1z_0z_0}^3F+u_x(-\partial_{z_0z_0}^2F+2\partial^3_{z_1xz_0}F) -\partial^2_{z_0x}F+\partial^3_{z_1xx}F=0.
     \end{aligned}
\end{equation}
We have the following.
\begin{theorem}
The equation \eqref{KdV-concreta} is equivalent to
\begin{equation}\label{KdVpara}
\begin{aligned}
u_t+\opbw(A(u))u+R_0=0,
\end{aligned}
\end{equation}
where 
\begin{equation*}
A(u):=\partial_{z_1z_1}^2F(\ii\xi)^3+\tfrac12\tfrac{d}{dx}\Big(\partial_{z_1z_1}^2F\Big)(\ii\xi)^2+
a_1(u,u_x,u_{xx},u_{xxx})(\ii\xi),
\end{equation*}
with $a_1$ real function and $R_0$  semi-linear remainder. Moreover we have the following estimates. Let $\sigma\geq s_0>1+1/2$ and consider $U,V\in H^{\sigma+3}_0$
\begin{align}
\|R_0(U)\|_{H^{\sigma}_0}&\leq C(\|U\|_{H^{s_0+3}_0})\|U\|_{H^{\sigma}_0},\quad \|R_0(U)\|_{H^{\sigma}_0}\leq C(\|U\|_{H^{s_0}_0})\|U\|_{H^{\sigma+3}_0}, \label{R01}\\
\|R_0(U)-R_0(V)\|_{H^{\sigma}_0}&\leq C(\|U\|_{H^{s_0+3}_0}+\|V\|_{H^{s_0+3}_0})\|U-V\|_{H^{\sigma}_0}\nonumber\\
&\quad+C(\|U\|_{H^{\sigma}_0}+\|V\|_{H^{\sigma}_0})\|U-V\|_{H^{s_0+3}_0},\label{R02}\\
\|R_0(U)-R_0(V)\|_{H^{s_0}_0}&\leq C(\|U\|_{H^{s_0+3}_0}+\|V\|_{H^{s_0+3}_0})\|U-V\|_{H^{s_0}_0},\label{R03}
\end{align}
where $C$ is a non decreasing and positive function. Concerning the para-differential operator we have for any $\s\geq 0$
\begin{equation}\label{paralip}
\|\opbw(A(u)-A(w))v\|_{H^{\s}_0}\leq C(\|u\|_{H^{s_0}_0}+ \|w\|_{H^{s_0}_0})\|u-w\|_{H^{s_0}_0}\|v\|_{H^{s_0+3}_0}.
\end{equation}
\end{theorem}
\begin{proof}
In the following we use the Bony paraproduct (Lemma \ref{Paraproduct}) and Prop. \ref{compo} and we obtain ($\tilde{R}_0$ is a smoothing remainder satisfying \eqref{R01}, \eqref{R02} and it possibly changes from line to line)
\begin{equation}
\begin{aligned}
 u_{xxx}\partial_{z_1z_1}^2 F&= \opbw(u_{xxx})\partial_{z_1z_1}^2F+\opbw(\partial_{z_1z_1}^2F)\circ\opbw((\ii\xi)^3)u+\tilde{R}_0\\
 					    &= \opbw(u_{xxx})\circ \opbw(\partial_{z_1z_1z_1}^3F)\circ\opbw(\ii\xi)u\\
					    &\quad +\opbw(\partial_{z_1z_1}^2F)\circ\opbw((\ii\xi)^3)u+\tilde{R}_0\\
					    &=\opbw(\partial_{z_1z_1}^2F(\ii\xi)^3)u+\tfrac32 \opbw(\tfrac{d}{dx}(\partial_{z_1z_1}^2F)\xi^2)u+\opbw(\tilde{a}_1(\ii\xi))+\tilde{R}_0,
 \end{aligned}
\end{equation}
where we have denoted by $\tilde{a}_1$ a real function depending on $x,u,u_x,u_{xx},u_{xxx}$. Analogously we obtain
						\begin{align}
2u_{xx}\partial^3_{z_1xz_1}F&= 2\opbw(\partial^3_{z_1z_1x}F(\ii\xi)^2)+\opbw(\tilde{a}_1(\ii\xi))+\tilde{R}_0,\\\label{ord2_2}
u_{xx}^2\partial_{z_1z_1z_1}^3F&=2\opbw(u_{xx}\partial^3_{z_1z_1z_1}F(\ii\xi^2))u+\opbw(\tilde{a}_1(\ii\xi))u+\tilde{R}_0,\\\label{ord2_3}
2u_xu_{xx}\partial_{z_1z_1z_0}^3F&=2\opbw\big(u_x\partial_{z_1z_1z_0}^3F(\ii\xi)^2)u+2\opbw(\tilde{a}_1(\ii\xi))u+\tilde{R}_0,
\end{align}
Summing up the previous equations we get
\begin{equation}\label{KdVpara}
\begin{aligned}
u_t&+\opbw(\partial_{z_1z_1}^2F(\ii\xi)^3)u+\\
&+\tfrac12\opbw(\tfrac{d}{dx}(\partial_{z_1z_1}^2F)(\ii\xi)^2)u+\opbw({a}_1(x,u,u_x,u_{xx},u_{xxx})\ii\xi)u+\tilde{R}_0(u)=0,
\end{aligned}\end{equation}
where ${a}_1$ is real and $R_0$ is a semi-linear remainder satisfying \eqref{R01} and \eqref{R02}.
\end{proof}

\section{Linear local well-posedness}

\begin{proposition}\label{lineare}
Let $s_0>1+1/2$,  $\Theta\geq r>0$, $u\in C^0([0,T];H_0^{s_0+3})\cap C^1([0,T]; H_0^{s_0})$ such that
\begin{equation}\label{teta-r}
\|u\|_{L^{\infty}H_0^{s_0+3}}+\|\partial_t u\|_{H_0^{s_0}}\leq \Theta, \quad \|u\|_{L^{\infty}H_0^{s_0}}\leq r.
\end{equation}
Let $\sigma\geq 0$  and $t\mapsto R(t)\in C^0([0,T],H^{\sigma}_0)$. The there exists an unique solution $v\in C^{0}([0,T];H^{\sigma}_0)\cap C^{1}([0,T];H^{\sigma-3}_0)$ of the linear inhomogeneous problem
\begin{equation}\label{KdVparalin}
\begin{aligned}
&v_t+\opbw(\partial_{z_1z_1}^2F(u,u_x)(\ii\xi)^3)v+\\
+&\tfrac12\opbw(\tfrac{d}{dx}(\partial_{z_1z_1}^2F(u,u_x))(\ii\xi)^2)v+\opbw(\tilde{a}_1(x,u,u_x,u_{xx},u_{xxx})(\ii\xi))v+R(t)=0,\\
&v(0,x)=v_0(x).
\end{aligned}\end{equation}
Moreover the solution satisfies the estimate 
\begin{equation}\label{gromwall}
\|v\|_{L^{\infty}H^{\sigma}_0}\leq e^{C_{\Theta}T}( C_r \|v_0\|_{H^{\sigma}_0}+C_{\Theta}T\|R\|_{L^{\infty}H^{\sigma}_0}).
\end{equation}
 \end{proposition}
 
 Consider the equation \eqref{KdVparalin}. We have for any $N\in\mathbb{N},$ $\sigma>1/2$ and $s\geq 0$
 \begin{equation}\label{simboli_stimati}
 \begin{aligned}
&\|\tilde{a}_1(x,u,u_x,u_{xx},u_{xxx})\|_{H^{\sigma}_0}\leq C(\|u\|_{H^{\sigma+3}_0})\\
& \|\tfrac{d}{dx}(\partial_{z_1z_1}^2F(u,u_x))\|_{H^{\sigma-1}_0}\leq C(\|u\|_{H^{\sigma+2}_0}) \\
&\|\partial_{z_1z_1}^2F(u,u_x)\|_{H^{\sigma}_0}\leq C(\|u\|_{H^{\sigma+1}_0}),\\
&|\partial_{z_1z_1}^2F(u,u_x)|\xi|^{2s}|_{2s,\sigma,N}\leq C_N(\|u\|_{H^{\sigma+1}_0}), \,\, |\tfrac{d}{dx}(\partial_{z_1z_1}^2F(u,u_x))(\ii\xi)^2|_{2,\sigma,N}\leq C_N(\|u\|_{H^{\sigma+2}_0}),\\
&|\tilde{a}_1(x,u_x,u_{xx},u_{xxx})|_{1,\sigma, N}\leq C_N(\|u\|_{H^{\sigma+2}_0}).
\end{aligned}
  \end{equation}
  For any $\epsilon>0$ we consider the regularised symbol 
  \begin{equation}\label{simboloregolare}
\begin{aligned} 
 \mathfrak{S}&^{\epsilon}(x,u,\xi):= \\
 &\Big(\partial_{z_1z_1}^2F(u,u_x)(\ii\xi)^3+\tfrac12\tfrac{d}{dx}(\partial_{z_1z_1}^2F(u,u_x))(\ii\xi)^2+\tilde{a}_1(u,u_x,u_{xx},u_{xxx})\ii\xi\Big)\chi(\epsilon \partial_{z_1z_1}^2F(u,u_x) \xi^3 ),
  \end{aligned}\end{equation}
  where $\chi$ is the same cut-off function defined after formula \eqref{seminorme}. We note that thanks to \eqref{simboli_stimati} and to the fact that the function $y\mapsto \xi^{\alpha}\partial_{\xi}^{\alpha}\chi(\epsilon y \xi^3)$ is bounded with its derivatives uniformly in $\epsilon\in (0,1)$ and $\xi\in\R$, the cut-off $\chi^{\epsilon}:=\chi(\epsilon \partial_{z_1z_1}^2F(u,u_x) \xi^3)$ satisfies for any $N\in\N$
  \begin{equation}\label{unieps}
  |\chi^{\epsilon}|_{0,\sigma,N}\leq C(\|u\|_{H^{\sigma+1}_0}), 
  \end{equation}
 uniformly in $\epsilon\in (0,1)$. \\
 In the following lemma we prove that, thanks to the Hamiltonian structure, we may eliminate the symbol of order two by means of a paradifferential change of variable. This term is the only one which has positive order and that is not skew-selfadjoint.
\begin{lemma}\label{key}
Define $\mathtt{d}(x,u,u_x):=\sqrt[6]{\partial^2_{z_1z_1}F(x,u,u_x)}$. Then we have  
\begin{equation}\label{parapa}
\begin{aligned}
\opbw(\mathtt{d})\circ&\opbw\Big(\chi\big(\epsilon \partial_{z_1z_1}^2F(u,u_x) \xi^3\big) \big(\partial^2_{z_1z_1}F(\ii\xi)^3+\tfrac12\tfrac{d}{dx}(\partial_{z_1z_1}^2F)(\ii\xi)^2\big)\Big)\circ \opbw(\mathtt{d}^{-1})v=\\
&\opbw\big(\chi\big(\epsilon \partial_{z_1z_1}^2F(u,u_x) \xi^3\big) \big[\partial^2_{z_1z_1}F(\ii\xi)^3+\tilde{a}_1(x,u,u_x,u_{xx},u_{xxx})(\ii\xi)\big]\big)v+R_0,\\
\end{aligned}\end{equation}
where $\tilde{a}_1$ is a real function and $R_0$ is a semilinear remainder verifying \eqref{R01}, \eqref{R02}, \eqref{R03}.
\end{lemma}
\begin{proof}
First of all the function $\mathtt{d}(x,u,u_x)$ is well defined because of hypothesis \eqref{ellit}.
We recall formula \eqref{cancelletto} (and the definition of the Poisson's bracket after \eqref{cancelletto}) and we denote $\chi_{\epsilon}:=\chi(\epsilon \partial_{z_1z_1}^2F(u,u_x) \xi^3)$. By using Theorem \ref{compo} with  $\rho\in(1,2]$ we obtain that the l.h.s. of the equation \eqref{parapa} equals
\begin{equation*}
\begin{aligned}
-\opbw&(\ii\chi_{\epsilon}\partial_{z_1z_1}F\xi^3)v-\frac12\opbw(\chi_{\epsilon}\tfrac{d}{dx}(\partial_{z_1z_1}^2F)\xi^2)v\\
&+3\opbw\Big(\chi_{\epsilon}\cdot\mathtt{d}^{-1}\cdot\tfrac{d}{dx} \mathtt{d}\cdot \partial_{z_1z_1}^2F\cdot\xi^2\Big)v+\opbw(\tilde{a}_1)+R_0,
\end{aligned}\end{equation*}
where $\tilde{a}_1$ is a purely imaginary function and $R_0$ a semilinear remainder.
One can verify the symbol of order two equals to zero by direct inspection.
\end{proof}
 
We introduce the smoothed version of the homogeneous part of \eqref{KdVparalin}, more precisely
\begin{equation}\label{KdVparalin-smooth}
\partial_tv^{\epsilon}=\opbw(\mathfrak{S}^{\epsilon}(x,u,u_x,u_{xx},u_{xxx};\xi))v^{\epsilon},
\end{equation} 
where  $\mathfrak{S}^{\epsilon}$ has been defined in \eqref{simboloregolare}. The operator is bounded, as a consequence for any $\epsilon>0$ there exists a unique solution of the equation \eqref{KdVparalin-smooth} which is $C^{2}([0,T],H^{\sigma}_0)$ for any $\sigma\geq 1$. Such an equation verifies \emph{a priori} estimates with constants independent of $\epsilon$.
 We have the following.
\begin{proposition}\label{energia}
Let $u$ be a function as in \eqref{teta-r}. For any $\sigma\geq 0$ there exist constants $C_{\Theta}$ and $C_r$, such that for any $\epsilon>0$ the unique solution of \eqref{KdVparalin} verifies
\begin{equation}\label{energiaene}
\|v^{\epsilon}\|_{H^{\sigma}_0}^2\leq C_r\|v_0\|_{H^{\sigma}_0}^2+C_{\Theta}\int_0^t\|v^{\epsilon}(\tau)\|_{H^{\s}_0}^2d\tau, \forall t\in [0,T].
\end{equation}
As a consequence we have
\begin{equation}\label{energiaene2}
\|v^{\epsilon}\|_{H^{\sigma}_0}\leq C_r e^{TC_{\theta}}\|v_0\|_{H^{\s}_0}, \forall t\in[0,T].
\end{equation}
\end{proposition}
 We define the modified energy
\begin{equation}\label{energia-modi} 
\begin{aligned}
\|&v\|_{\s,u}^2:=\\
&\left\langle\opbw\left((\partial^2_{z_1z_1}F(x,u,u_x))^{\frac23 \s}|\xi|^{2\sigma}\right)\opbw\left(\mathtt{d}(x,u,u_x)\right)v,\opbw\left(\mathtt{d}(x,u,u_x)\right)v\right\rangle_{L^2},
\end{aligned}
\end{equation}
where $\langle\cdot,\cdot\rangle$ is the standard scalar product on $L^2(\R)$ and $\mathtt{d}$ is defined in Lemma \ref{key}, note that the function $(\partial^2_{z_1z_1}F(x,u,u_x))^{\frac23 \s}$ is well defined for any $\sigma\in\R$ thanks to \eqref{ellit}.\\
In the following we prove that $\|\cdot\|_{\s,u}$ is equivalent to $\|\cdot\|_{H_0^{\s}}$.
\begin{lemma}
Let $s_0>1/2,\sigma\geq 0, r\geq0$. The there exists a constant (depending on $r$ and $\sigma$) such that for any $u$ such that $\|u\|_{H^{s_0}_0}\leq r$ we have
\begin{equation}\label{equivalenza}
C_r^{-1}\|v\|_{H^{\sigma}_0}^2-\|v\|^2_{H^{-3}_0}\leq\|v\|_{\sigma,u}^2\leq C_r\|v\|_{H^{\sigma}_0}^2\end{equation}
for any $v$ in ${H^{\sigma}_0}$.
\end{lemma}
\begin{proof}
Concerning the  second inequality in \eqref{equivalenza}, we reason as follows. We have
\begin{equation*}\begin{aligned}
\|v\|_{\sigma,u}^2\leq& \|\opbw((\partial^2_{z_1z_1}F(x,u,u_x))^{\frac23 \s}|\xi|^{2\sigma})\opbw(\mathtt{d}(x,u,u_x))v\|_{H_0^{-\s}}\\
&\times\|\opbw(\mathtt{d}(x,u,u_x))v\|_{H^{\s}_0}\\
\leq& C_r\|v\|_{H^{\s}_0},
\end{aligned}\end{equation*}
where in the last inequality we used Theorem \ref{azione} and the fact that $\mathtt{d}$ is a symbol of order zero.
We focus on the first inequality in \eqref{equivalenza}. Let $\delta>0$ be such that $s_0-\delta=1/2$, then applying Theorem \ref{compo} with $s_0=\delta$ instead of $s_0$ and $\rho=\delta$, we have 
\begin{equation}\label{fuffetta1}
\begin{aligned}
&\opbw((\partial^2_{z_1z_1}F(x,u,u_x))^{\frac13 \s})\circ\opbw(|\xi|^{2\s})\circ\opbw((\partial^2_{z_1z_1}F(x,u,u_x))^{\frac13 \s})\\
&=\opbw(\opbw((\partial^2_{z_1z_1}F(x,u,u_x))^{\frac23 \s}|\xi|^{2\s})+\mathcal{R}^{2\s-\delta}(u),
\end{aligned}
\end{equation}
where
\begin{equation*}
\|\mathcal{R}^{2\s-\delta}(u)f\|_{H_0^{\bar\s-2\sigma+\delta}}\leq C(r,\bar\s)\|f\|_{H^{\bar\s}_0}.
\end{equation*}
Analogously we obtain
\begin{equation}\label{fuffetta2}\begin{aligned}
\opbw&((\partial^2_{z_1z_1}F(x,u,u_x))^{-\frac13 \s})\circ\opbw(\mathtt{d}^{-1}(x,u,u_x,u_xx))\circ\\
&\opbw((\partial^2_{z_1z_1}F(x,u,u_x))^{\frac13 \s})\circ\opbw(\mathtt{d}(x,u,u_x,u_xx))\\
&=1+R^{-\delta}(u),
\end{aligned}\end{equation}
where
\begin{equation*}
\|\mathcal{R}^{-\delta}(u)f\|_{H_0^{\bar \s}}\leq C(r,\bar\s)\|f\|_{H^{\bar\s-\delta}_0},
\end{equation*}
for any $f$ in $H^{\bar\s-\delta}_0$.
Therefore we have
\begin{equation*}
\begin{aligned}
&\|v\|_{H^{\sigma}_0}^2\stackrel{\eqref{fuffetta2}}{\lesssim} \\
&\|\opbw((\partial^2_{z_1z_1}F(x,u,u_x))^{-\frac13 \s})\opbw(\mathtt{d}^{-1})\opbw(\partial^2_{z_1z_1}F(x,u,u_x))^{\frac13 \s})\opbw(\mathtt{d})v\|_{H^{\s}_0}^2+\|v\|_{H^{\s-\delta}_0}^2\\
\leq& C_r(\|\opbw(\partial^2_{z_1z_1}F(x,u,u_x))^{\frac13 \s})\opbw(\mathtt{d})v\|_{H^{\s}_0}^2+\|v\|_{H^{\s-\delta}_0}^2)\\
\stackrel{\eqref{fuffetta1}}{=}& C_r(\|v\|_{u,\sigma}^2+\|v\|_{H^{\s-\delta/2}_0}^2+\|v\|_{H^{\s-\delta}_0}^2).
\end{aligned}\end{equation*}
Then by using the interpolation inequality $\|f\|_{H^{\theta s_1+(1-\theta)s_2}_0}\leq \|f\|_{H^{ s_1}_0}^{\theta}\|f\|_{H^{ s_2}_0}^{1-{\theta}}$ which is valid for any $s_1<s_2$, $\theta\in[0,1]$ and $f\in H^{s_2}$, we get (by means of the Young inequality $ab\leq p^{-1}a^p+q^{-1}b^q$, with $1/p+1/q=1$ and $p=2(\s+3)/\delta,$ $q=2(\sigma+3)/[2(\sigma+3)-\delta]$)
\begin{equation*}
\|v\|_{H^{\s-\delta/2}_0}^2\leq (\|v\|_{H^{-3}_0}^2)^{\frac{\delta}{2}\frac{1}{\s+3}}(\|v\|_{H^{\s}_0}^2)^{\frac{2(\s+3)-\delta}{2(\s+3)}}\leq \tfrac{\delta}{2(\s+3)}\|v\|_{H^{-3}_0}^2\tau^{-\frac{2(\s+3)}{\delta}}+\tfrac{2(\s+3)-\delta}{2(\s+3)}\tau^{\frac{2(\s+3)-\delta}{2(\s+3)}}\|v\|_{H^{\s}_0}^2,
\end{equation*}
for any $\tau>0$. Choosing $\tau$ small enough we conclude.\end{proof}
We are in position to prove Prop. \ref{energia}.
\begin{proof}[proof of Prop. \ref{energia}]
We take the derivative with respect to $t$ of the modified energy \eqref{energia-modi} along the solution $v^{\epsilon}$ of the equation \eqref{KdVparalin-smooth}. We have
\begin{align}
\tfrac{d}{dt}\|v^{\epsilon}\|_{\sigma,u}=&\langle\opbw\left(\tfrac{d}{dt}(\partial^2_{z_1z_1}F)^{\frac23 \s}|\xi|^{2\sigma}\right)\opbw\left(\mathtt{d}\right)v^{\epsilon},\opbw\left(\mathtt{d}\right)v^{\epsilon}\rangle_{L^2}\label{f_1}\\
+&\langle\opbw\left((\partial^2_{z_1z_1}F)^{\frac23 \s}|\xi|^{2\sigma}\right)\opbw\left(\tfrac{d}{dt}\mathtt{d}\right)v^{\epsilon},\opbw\left(\mathtt{d}\right)v^{\epsilon}\rangle_{L^2}\label{f_2}\\
+&\langle\opbw\left((\partial^2_{z_1z_1}F)^{\frac23 \s}|\xi|^{2\sigma}\right)\opbw\left(\mathtt{d}\right)\tfrac{d}{dt}v^{\epsilon},\opbw\left(\mathtt{d}\right)v^{\epsilon}\rangle_{L^2}\label{f_3}\\
+&\langle\opbw\left((\partial^2_{z_1z_1}F)^{\frac23 \s}|\xi|^{2\sigma}\right)\opbw\left(\mathtt{d}\right)v^{\epsilon},\opbw\left(\tfrac{d}{dt}\mathtt{d}\right)v^{\epsilon}\rangle_{L^2}\label{f_4}\\
+&\langle\opbw\left((\partial^2_{z_1z_1}F)^{\frac23 \s}|\xi|^{2\sigma}\right)\opbw\left(\mathtt{d}\right)v^{\epsilon},\opbw\left(\mathtt{d}\right)\tfrac{d}{dt}v^{\epsilon}\rangle_{L^2}\label{f_5}.
\end{align}
The most important term, where we have to see a cancellation, is the one given by \eqref{f_3}+\eqref{f_5}. Using the equation  \eqref{KdVparalin-smooth} we deduce that  \eqref{f_3}+\eqref{f_5} equals to 
\begin{align*}&\langle\opbw\left((\partial^2_{z_1z_1}F)^{\frac23 \s}|\xi|^{2\sigma}\right)\opbw\left(\mathtt{d}\right)v^{\epsilon},\opbw\left(\mathtt{d}\right)\opbw(\mathfrak{S})v^{\epsilon}\rangle_{L^2}+\\
&\langle\opbw\left((\partial^2_{z_1z_1}F)^{\frac23 \s}|\xi|^{2\sigma}\right)\opbw\left(\mathtt{d}\right)\opbw(\mathfrak S)v^{\epsilon},\opbw\left(\mathtt{d}\right)v^{\epsilon}\rangle_{L^2}\
\end{align*}
 where $\mathfrak{S}:=\mathfrak{S}^{\epsilon}$ has been defined in \eqref{simboloregolare}. We note that by using Theorem \ref{compo} with $\rho=3$ we obtain 
 $$\opbw(\mathtt{d}^{-1})\opbw(\mathtt{d})v^{\epsilon}=v^{\epsilon}+\mathcal{R}^{-3}(u)v^{\epsilon},$$
  where $\mathcal{R}^{-3}$ verifies \eqref{resto} with $\rho=3$. We plug this identity in the previous equation and we note that the contribution coming from $\mathcal{R}^{-3}$ is bounded by $C_r\|v^{\epsilon}\|_{H^{\s}_0}^2$ thanks to  Theorems \ref{compo}, \ref{azione}, to the Cauchy Schwartz inequality and to the assumption \eqref{teta-r}.We are left with 
 \begin{align*}&\langle\opbw\left((\partial^2_{z_1z_1}F)^{\frac23 \s}|\xi|^{2\sigma}\right)\opbw\left(\mathtt{d}\right)v^{\epsilon},\opbw\left(\mathtt{d}\right)\opbw(\mathfrak{S})\opbw(\mathtt{d}^{-1})\opbw(\mathtt{d})v^{\epsilon}\rangle_{L^2}+\\
&\langle\opbw\left((\partial^2_{z_1z_1}F)^{\frac23 \s}|\xi|^{2\sigma}\right)\opbw\left(\mathtt{d}\right)\opbw(\mathfrak S)\opbw(\mathtt{d}^{-1})\opbw(\mathtt{d})v^{\epsilon},\opbw\left(\mathtt{d}\right)v^{\epsilon}\rangle_{L^2}.
\end{align*}
At this point we are ready to use Lemma \ref{key} and we obtain that the previous quantity equals
 \begin{align*}&\langle\opbw\left((\partial^2_{z_1z_1}F)^{\frac23 \s}|\xi|^{2\sigma}\right)\opbw\left(\mathtt{d}\right)v^{\epsilon},\opbw\left(\chi^{\epsilon}\cdot(\partial_{z_1z_1}^2F(\ii\xi)^3+\tilde{a}_1(\ii\xi))\right)\opbw(\mathtt{d})v^{\epsilon}\rangle_{L^2}+\\
&\langle\opbw\left((\partial^2_{z_1z_1}F)^{\frac23 \s}|\xi|^{2\sigma}\right)\opbw\left(\chi^{\epsilon}\cdot(\partial_{z_1z_1}^2F(\ii\xi)^3+\tilde{a}_1(\ii\xi))\right)\opbw(\mathtt{d})v^{\epsilon},\opbw\left(\mathtt{d}\right)v^{\epsilon}\rangle_{L^2}.
\end{align*}
By using the skew self-adjoint character of the operators, we deduce that the main term to estimate is the commutator
\begin{equation}\label{commutatore}
\Big[\opbw\left((\partial^2_{z_1z_1}F)^{\frac23 \s}|\xi|^{2\sigma}\right), \opbw\left(\chi^{\epsilon}\cdot(\partial_{z_1z_1}^2F(\ii\xi)^3+\tilde{a}_1(\ii\xi))\right)\Big] \opbw(\mathtt{d})v^{\epsilon}.
\end{equation}
We start from the first summand. By using Theorem \ref{compo} and Remark \ref{simmetrie} with $\rho=3$ we obtain that 
\begin{align*}
C:=\Big[&\opbw\left((\partial^2_{z_1z_1}F)^{\frac23 \s}|\xi|^{2\sigma}\right), \opbw\left(\chi^{\epsilon}\cdot(\partial_{z_1z_1}^2F(\ii\xi)^3\right))\Big]\opbw(\mathtt{d})v^{\epsilon}=\\
&\frac{1}{\ii}\opbw\Big(\Big\{(\partial^2_{z_1z_1}F)^{\frac23 \s}|\xi|^{2\sigma},\chi^{\epsilon}\cdot\partial_{z_1z_1}^2F(\ii\xi)^3\Big\}\Big)\opbw(\mathtt{d})v^{\epsilon}+\mathcal{R}^0(u)\opbw(\mathtt{d})v^{\epsilon}.
\end{align*}
We recall that the cut-off function has been chosen as $\chi^{\epsilon}:=\chi(\epsilon\partial_{z_1z_1}^2F\xi^3)$, so that we obtain, by direct inspection that the Poisson bracket above equals to $0$.  Recalling that $\mathtt{d}$ is a symbol of order $0$, by using also Theorem \ref{azione} and the assumption \eqref{teta-r}, we may  obtain the bound $\langle C,\opbw(\mathtt{d})v^{\epsilon}\rangle\leq C_r\|v^{\epsilon}\|_{H^{\sigma}_0}^2$. \\
The second summand, i.e. the one coming from $\tilde{a}_1(\ii\xi)$ in \eqref{commutatore}, may be treated in a similar way: one uses Theorem \ref{compo} with $\rho=1$, at the first order the contribution is equal to zero, then the remainder is a bounded operator from $H^{2\sigma}_0$ to $H^{0}_0$ and one concludes as before, by using also the duality inequality  $\langle f,g\rangle_{L^2}\leq \|f\|_{H^{-\s}_0}{\|g\|_{H^{\s}_0}}$, bounding everything by  $C_r\|v^{\epsilon}\|_{H^{\sigma}_0}^2$. We are left with \eqref{f_1}, \eqref{f_2} and $\eqref{f_4}$. These terms are simpler, one just has to use the duality  inequality recalled above, then Theorem \ref{azione} and the fact that 
\begin{equation*}
|\tfrac{d}{dt} \mathtt{d}(x,u,u_x)|_{0,\s,{4}}, \quad |\tfrac{d}{dt}(\partial^2_{z_1z_1}F)^{\frac23 \s}|_{0,0,{4}}\leq C_{\Theta}\|u\|_{H^{\sigma}_0},
\end{equation*}
where we have used the first one of the assumptions \eqref{teta-r}. We eventually obtained $\tfrac{d}{dt}\|v^{\epsilon}\|_{\sigma,u}^2\leq C_{\Theta}\|v^{\epsilon}\|^2_{H^{\sigma}_0}$, integrating ove the time interval $[0,t)$ we obtain
\begin{equation*}
\|v^{\epsilon}\|_{\sigma,u(t)}^2\leq \|v^{\epsilon}(0)\|_{\sigma,u(0)}^2+C_{\Theta}\int_0^{t}\|v^{\epsilon}(\tau)\|_{H^{\sigma}_0}^2d\tau\leq C_r \|v^{\epsilon}(0)\|_{H^{\sigma}_0}^2+C_{\Theta}\int_0^{t}\|v^{\epsilon}(\tau)\|_{H^{\sigma}_0}^2d\tau
\end{equation*}
We now use \eqref{equivalenza} and the fact that $\|\partial_{t}v^{\epsilon}\|_{H^{-3}_0}\leq C_{\Theta}\|v^{\epsilon}\|_{H^{0}_{0}}\leq C_{\Theta} \|v^{\epsilon}\|_{H^{\sigma}_0}$ since $\sigma\geq 0.$
\end{proof}
We may now prove Prop. \ref{lineare}.
\begin{proof}[proof of Prop. \ref{lineare}]
Let $v_0$ be in $C^{\infty}(\T)$,  then the sequence $v^{\epsilon}$ given by Prop. \ref{energia} converges, thanks to the theorem of Ascoli-Arzel\`a, to a solution $v\in C^{0}([0,T);H^{\s}_0)$ of the equation \eqref{KdVparalin} with $R(t)=0$. When the initial condition is just in $H^{\s}_0$ one can use classical approximation arguments. The flow $\Phi(t)$ of the equation \eqref{KdVparalin} with $R(t)=0$ is well defined as a bounded operator form $H^{\s}_0$ to $H^{\s}_0$ and satisfies the estimate
\begin{equation*}
\|\Phi(t)v_0\|_{H^{\s}_0}\leq C_re^{C_{\Theta}t}\|v_0\|_{H^{\s}_0}.
\end{equation*}
One concludes by using the Duhamel formulation of \eqref{KdVparalin}.
\end{proof}
\section{Nonlinear local well posedness}
To build the  solutions of the nonlinear problem \eqref{KdVparalin}, we shall consider a classical quasi-linear iterative scheme, we follow the approach in \cite{FI1,FI2,BMM1, Mietka}.
 Set 
$$A(u):=\opbw\left(\partial_{z_1z_1}^2F(u,u_x)(\ii\xi)^3+\tfrac12 \tfrac{d}{dx}(\partial_{z_1z_1}^2F(u,u_x))(\ii\xi)^2+\tilde{a}_1(x,u,u_x,u_{xx},u_{xxx})(\ii\xi)\right)$$
and define
\begin{align*}
\mathcal{P}_1:& \quad \partial_t u_1= A(u_0)u_1;\\
\mathcal{P}_n:& \quad \partial_t u_n=A(u_{n-1})u_n+R(u_{n-1}), \quad n\geq 2.
\end{align*}
The proof of the main Theorem \ref{totale} is a  consequence of the next lemma. Owing to such a lemma one can follow closely the proof of Lemma 4.8 and Proposition 4.1 in \cite{BMM1} or the proof of Theorem 1.2 in \cite{FI2}(this is the classical Bona-Smith technique \cite{BSkdv}, but we followed the notation of \cite{BMM1, FI2}). We do not reproduce here such a proof.
\begin{lemma}
Let $s>\frac{1}{2}+4$. Set $r:=\|u_0\|_{s_0}$ and $s_0>1+1/2$. There exists a time $T:=T(\|u_0\|_{H^{s_0+3}})$ such that for any $n\in \mathbb{N}$ the following statements are true. \\
$(\mathcal{S}{0})_n$: There exists a unique solution $u_n$ of the problem $\mathcal{P}_n$ belonging to the space $C^{0}([0,T);H^{s})\cap C^{1}([0,T);H^{s-3})$.\\
$(\mathcal{S}{1})_n$: There exists a constant $\mathtt{C}_r\geq 1$ such that if $\Theta=4\mathtt{ C}_r\|u_0\|_{H^{s_0+3}_0}$ and $M=4 \mathtt{C}_r \|u_0\|_{H^{s}_0}$, for any $1\leq m\leq n$, for any $1\leq m\leq n$ we have 
\begin{align}
&\|u_m\|_{L^{\infty}H^{s_0}_0}\leq \mathtt{C}_r,\label{caz1}\\
&\|u_m\|_{L^{\infty}H^{s_0+3}_0}\leq \Theta,\quad \|\partial_tu_m\|_{L^{\infty}H^{s_0}_0}\leq \mathtt{C}_r\Theta,\label{caz2}\\
&\|u_m\|_{L^{\infty}H^{s}_0}\leq M,\quad \|\partial_tu_m\|_{L^{\infty}H^s_0}\leq \mathtt{C}
_rM.\label{caz3}
\end{align}
$(\mathcal{S}{2})_n$: For any $1\leq m\leq n$ we have 
\begin{equation}
\|u_1\|_{L^{\infty}H^{s_0}_0}\leq \mathtt{C}_r,\quad \|u_{m}-u_{m-1}\|_{L^{\infty}H^{s_0}_0}\leq2^{-m}\mathtt{C}_r, \quad m\geq 2.
\end{equation}
\end{lemma}
\begin{proof}
We proceed by induction over $n$. We prove $(S0)_1$, by using Proposition \ref{lineare} with $R(t)=0$, $u\rightsquigarrow u_0$ and $v\rightsquigarrow u_1$; we obtain a solution $u_1$ which is defined on every interval $[0,T)$ and verifies the estimate $\|u_1\|_{L^{\infty}_TH^{\s}_0}\leq e^{T\|u_0\|_{H^{\s}_0}}C_r\|u_0\|_{H^{\s}_0}$, $\s\geq 0$ with $C_r>0$ given by Proposition \ref{lineare}. $(S1)_1$ is a consequence of the previous estimate applied with $\s=s_0$ for \eqref{caz1} and \eqref{caz2}, with $\s=s$ for \eqref{caz3}. In order to obtain the seconds in \eqref{caz2} and \eqref{caz3},  one has to fix $T\leq 1/\|u_0\|_{s_0}$  and  use the equation for $u_1$ together with Theorem \ref{azione} and one finds $M$ which depends on $\|u_0\|_{H^s_0}$ and $\Theta$ which depends on $\|u_0\|_{H^{s_0}_0}$ and on a constant $C_r$ depending only on $\|u_0\|_{s_0}$.  $(S2)_1$ is trivial. 

We assume that $(SJ)_{n-1}$ holds true for any $J=0,1,2$ and we prove that $(SJ)_n$.\\
Owing to $(S0)_{n-1}$ and $(S1)_{n-1}$, the $(S0)_n$ is a direct consequence of Proposition \ref{lineare}. 
Let us prove \eqref{caz1} with $m=n$. By using \eqref{gromwall} applied to the problem solved by $u_n$, the estimate \eqref{R01} with $\s=s_0$, \eqref{caz1} with $m=n-1$ and $(S0)_{n-1}$, we obtain $\|u_n\|_{L^{\infty} H^{s_0}_0}\leq e^{C_{\Theta}T}(C_r\|u_0\|_{H^{s_0}_0}+\mathtt{C}_rC_{\Theta}T)$, the thesis follows by choosing $e^{C_{\Theta}T}C_\Theta T<1/4$ and $\mathtt{C}_r\geq \|u_0\|_{H^{s_0}_0}/4C_r$.\\
We prove the first in \eqref{caz2}. Applying \eqref{gromwall} with $\sigma=s_0+3$ and $v\rightsquigarrow u_n$, $u\rightsquigarrow u_{n-1}$, the estimate on the remainder \eqref{R01} and using $(S1)_{n-1}$ we obtain $\|u_n\|_{s_0+3}\leq e^{C_{\Theta}T}\mathtt{C}_r\|u_0\|_{s_0+3}+\Theta C_{\Theta}Te^{C_{\Theta}T}$, fixing $T$ small enough such that $TC_{\Theta}\leq1$ and $TC_{\Theta}e^{C_{\Theta}T}\leq 1/4,$ the thesis follows from the definition $\Theta:=4\mathtt{C}_r\|u_0\|_{H^{s_0}_0}$. The second in \eqref{caz2} may be proven by using the equation for $u_n$ and the second in \eqref{R01}
\begin{equation*}
 \|\partial_t u_n\|_{H^{s_0}_0}\leq\|A(u_{n-1})u_n\|_{H^{s_0}_0}+\|R(u_{n-1})\|_{H^{s_0}_0}\leq C(\|u_{n-1}\|_{H^{s_0}_0})\|u_{n}\|_{H^{s_0+3}_0}\leq \mathtt{C}_r\Theta.
\end{equation*}
The \eqref{caz3} is similar. We prove $(S2)_n$, we write the equation solved by $v_n=u_n-u_{n-1}$
\begin{equation*}
\partial_t v_n=A(u_{n-1})v_n+f_n,\quad f_{n}=\big[A(u_{n-1})-A(u_{n-2})\big]u_{n-1}+R(u_{n-1})-R(u_{n-2}).
\end{equation*}
By using \eqref{R02}, \eqref{paralip} and the $(S2)_{n-1}$ we may prove that $\|f_n\|_{H^{s_0}_0}\leq C_{\Theta}\|v_{n-1}\|_{H^{s_0}_0}$. We apply again Proposition \ref{lineare} with $\s=s_0$ and we find $\|v_n\|_{H^{s_0}_0}\leq C_{\Theta}Te^{C_{\Theta}T}\|v_{n-1}\|_{H^{s_0}_0}$, as $T$ has been chosen small enough we conclude the proof.
\end{proof}
\def\cprime{$'$}


\begin{thebibliography}{10}




%
%
%
%
%

\bibitem{BMM1}
M. Berti, A. Maspero and F. Murgante. 
\newblock{Local well posedness of the Euler-Korteweg
equations on $\mathbb{T}^{d}$ }. 
\newblock{\em Journal of Dynamics and Differential equations}, 33:1475-1513, 2021.

\bibitem{BSkdv}
{J. L. Bona and R. Smith},
\newblock{The initial-value problem for the {K}orteweg-de {V}ries equation.}
\newblock{\em {Philos. Trans. Roy. Soc. London Ser. A}} {278}:555-604, {1975}.


\bibitem{Bourgain}
J. Bourgain
\newblock{{F}ourier transform restriction phenomena for certain lattice subsets and applications to nonlinear evolution equation II: {T}he {K}d{V} equation.}
\newblock{\em Geom. Fun. Anal.}, 3:209-262, 1993.



%
%
%

%
%
%
%
%
%
%

%
%
%
%
%
%
%

\bibitem{CCT}
M. Christ, J. Colliander and T. Tao
\newblock{Asymptotics, frequency modulation, and low regularity ill-posedness for canonical defocusing equations}
\newblock{\em Am. j. math.}, 125:1235-1293, 2003.

\bibitem{CKSTT}
 J. Colliander, M. Keel, G. Staffilani, H. Takaoka and T.Tao,
\newblock{Sharp global well-posedness for {K}d{V} and modified {K}d{V} on $\mathbb{R}$ and $\mathbb{T}$}
\newblock{\em J. Am. math. soc.}, 16:7-5-749, 2003.



\bibitem{CKS}
{W. Craig, T. Kappeler and W. A. Strauss}
\newblock{Gain of regularity for equations of KdV type.}
\newblock{\em Ann. Inst. Henri Poincaré, Anal. Non Lin\'eaire} 9, no. 2, p. 147–186, 1992.
%
%
%
%
%
%
%
%
%
%
%
%
%
%
%

%
%
%
%
\bibitem{FGI20}
{R. Feola, B. Gr\'ebert and F. Iandoli}.
\newblock{Long time solutions for quasi-linear 
Hamiltonian perturbations of Schr\"odinger 
and Klein-Gordon equations on tori}. 
\newblock{\emph{To appear on Analysis and PDE}}, 2020.



\bibitem{FI1}
R.~Feola and F.~Iandoli.
\newblock {L}ocal well-posedness for quasi-linear {N}{L}{S} with large {C}auchy
  data on the circle.
\newblock {\em Annales de l'Institut Henri Poincare (C) Analyse non
  lin\'eaire}, 36(1):119--164, 2018.
%
 \bibitem{FI2}
R.~Feola and F.~Iandoli.
\newblock {L}ocal well-posedness for the {H}amiltonian quasi-linear {S}chr\"odinger equation on tori.
\newblock {\em Journal des Math. Pures et Appliquées}, 147:243-281, 2022.

\bibitem{FIM} R. Feola, F.Iandoli and F. Murgante.
\newblock{Long-time stability of the quantum hydrodynamic system on irrational tori}.
\emph{Math. Ing.},4 (3):1-24, 2022.

%
%
%
%
%
%





%
%
%

 \bibitem{IP} A.D. ~Ionescu and F. ~Pusateri.
{Long-time existence for multi-dimensional periodic water waves}. 
{\emph{Geom. Funct. Anal.}}, 29: 811--870, 2019. 




%
%

\bibitem{kato}
T.~Kato.
\newblock {\em {S}pectral Theory and Differential Equations. {L}ecture Notes in
  Mathematics, (eds.) Everitt, W. N.}, volume 448, chapter ``{Q}uasi-linear
  equations evolutions, with applications to partial differential equations''.
\newblock Springer, Berlin, Heidelberg, 1975.

\bibitem{KPV17}
C.~E. Kenig, G.~Ponce, and L.~Vega.
\newblock{The Cauchy problem for the {K}orteweg-de {V}ries equation in {S}obolev spaces of negative indices.}
\newblock{\em Duke Math. j. }, 71:1-21, 1993




\bibitem{KPV19}
C.~E. Kenig, G.~Ponce, and L.~Vega.
\newblock{A bilinear estimate with applications to the {K}d{V} equation}
\newblock{\em J. Am. Math. Soc.}, 9:573-603,1996.




\bibitem{KPV18}
C.~E. Kenig, G.~Ponce, and L.~Vega.
\newblock{Well-posedness and scattering results for the generalized {K}orteweg-de {V}ries equation via the contraction principle}
\newblock{\em Commun. Pure Appl. Math.}, 46:527-620,1993.


%
%
%

%
%
%
%
%
%
%
%
%

\bibitem{Mietka}
C. Mietka.
\newblock{{{O}n the well-posedness of a quasi-linear {K}orteweg-de {V}ries equation}}.
\newblock{\em Annales Math. Blaise Pascal, 24: 83-114, 2017.}

%
%
%
%
%
%
%
%


%







\end{thebibliography}
\end{document}